\documentclass[11pt,leqno]{amsart}
\usepackage{amsfonts, amsmath, amsthm, amssymb,xspace}
\usepackage[breaklinks=true]{hyperref}

\theoremstyle{plain}
\newtheorem{theorem}{Theorem}[section]
\newtheorem{lemma}{Lemma}[section]
\newtheorem{prop}{Proposition}[section]

\theoremstyle{definition}
\newtheorem{defin}{Definition}[section]

\newcommand{\bggo}{\mathcal O}

\newcommand{\mf}[1]{\displaystyle{\mathfrak{#1}}}

\newcommand{\comment}[1]{}

\DeclareMathOperator{\Gr}{\ensuremath{gr}}

\DeclareMathOperator{\ad}{\ensuremath{ad}}

\DeclareMathOperator{\Sym}{\ensuremath{Sym}}

\begin{document}

\title{On Infinitesimal Cherednik algebras of $\mf{gl}_2$}
\author{ Akaki Tikaradze}

\address{The University of Toledo \hfill\newline Department of Mathematics
\hfill\newline Toledo, Ohio, USA \hfill\newline
e-mail: {\tt tikar@math.uchicago.edu}}

\begin{abstract}
We prove that the center of an infinitesimal Cherednik algebra of
$\mf{gl}_2$ is the polynomial algebra of two variables over the field of characteristic 0. 
In positive characteristic we show that any infinitesimal Cherednik algebra is a finitely generated module
over its center.
\end{abstract}

\maketitle
\section{Introduction}

For a given reductive algebraic group $G$ (over a field $k$) and its finite dimensional
representation $V$, Etingof, Gan, and Ginzburg ([EGG]) introduced a new family of algebras 
called continuous Hecke algebras and infinitesimal Hecke algebras as certain deformations
of the semi-direct product algebras $O(G)^*\ltimes TV$ and $\mf{U}\mf{g}\ltimes TV$ respectively.
 Here $O(G)^{*}$
is the dual of the ring of regular functions on $G,$ where multiplication
in $O(G)^{*}$ is defined by the convolution and $TV$ is the
the tensor algebra of $V$.
If $G$ is a finite group, then these algebras are Drinfeld's degenerate affine Hecke algebras
which include widely studied rational Cherednik algebras [EG]. Let us recall the precise definition
of infinitesimal Hecke algebras [EGG].

  Let $V$ be a finite dimensional module over a reductive Lie algebra $\mf{g}$ and let $\gamma:V\otimes V\to \mf{U}\mf{g}$ be a $\mf{g}$-invariant skew-symmetric pairing. Then one considers the algebra
  $H_\gamma$ defined as the quotient of $\mf{U}\mf{g}\ltimes TV$  by the relations
  $$[v, w]=\gamma(v, w),  v, w\in V.$$ There is a natural algebra filtration
on $H_{\gamma}$ defined by setting $deg(\alpha)=0, deg(v)=1$ for $\alpha\in\mf{Ug}, v\in V.$
It is clear there is a natural map  
$\mf{U}(\mf{g}\ltimes V)\to gr(H_{\gamma})$ which is a surjactive graded algebra homomorphism.
If this map is an isomorphism
then $H_{\gamma}$ is said to be an infinitesimal Hecke algebra. There is a special class of infinitesimal
Hecke algebras called infinitesimal Cherednik algebras, where one takes 
$V=k^n\oplus (k^n)^{*}$ and $\mf{g}\subset \mf{gl}_n$
with the natural action on $V$ and it is required that $\gamma(v, w)=0$ if $v, w\in k^n,$ or 
$v, w\in (k^n)^{*}.$ 


In this paper, we prove that for the case $n=2$ and
$\mf{g}=\mf{gl}_2,$ the center of $H_{\gamma}$ is isomorphic to the polynomial algebra in two variables and 
$gr(\mf{Z}(H_{\gamma}))$=$\mf{Z}(\mf{U}(\mf{gl}_2\ltimes V))$ (for arbitrary algebra $B$, we will denote its center by $\mf{Z}(B)$). We have proved a similar result for infinitesimal Hecke algebras of $\mf{sl}_2$ and $V=k^2$ in 0 and sufficiently large characteristic ([KT], [T1], [T2]). We also establish that if the ground field has positive characteristic, then for any Cherednik algebra $H_{\gamma}$ (for arbitrary $\mf{g}$ and $n$), its center is big, meaning that $H_{\gamma}$ is a finitely generated module over its center.
\section{The center}
From now on, $\mf{g}=\mf{gl}_2, n=2$ and char($k$)=0.
Let us now write a more explicit description of the algebra $H_{\gamma}.$ We have the 
basis $k^2 = k x \oplus k y$ 
\\and $(k^2)^{*} = k x_1 \oplus k y_1$ of
$k^2$ and $(k^2)^{*}$ respectively, and $\mf{gl}_2 = \mf{sl}_2 \oplus k
\tau$, where $\tau$ is the (central) identity matrix. Here we adopt
the convention that $e, f, h$ denote the standard basis elements of $\mf{sl}_2$ and
$$[e, x_1]=[e, x]=0, [f, x]=y, [f, x_1]=y_1, [\tau,x]=x, [\tau, x_1]=y_1.$$
In particular, $x, x_1$ have weight 1 (with respect to ad($h$) action) and $k^2$ has weight 1 with respect
to $\tau$, and $(k^2)^{*}$ has weight -1. 

It is shown in [EGG] that there exists an element $c\in \mf{Z}\mf{Ug}$ such that $y_1x-x_1y-c$ belongs
to the center of $H_{\gamma}.$ Thus, element $c$ completely determines $\gamma,$ so for this reason
we are going to denote $H_{\gamma}$ by $H_c.$
Therefore, $H_c$ is a quotient of $\mf{U}\mf{g}\ltimes T(V)$ by the relations
$$[x, y]=0=[x_1, y_1], [y_1x-x_1y, x]=[c, x]$$
 and we assume that gr$(H_c)=\mf{U}(\mf{g}\ltimes V)$
(thus $H_c$ is an infinitesimal Cherednik algebra). 

Remark that $H_c$ possesses an anti-involution $\eta :H_c\to H_c$ defined as follows [KT].
$$\eta(x)=y_1, \eta(x_1)=-y, \eta(e)=f, \eta(h)=h, \eta(\tau)=\tau.$$
The following is the main result of the paper.

\begin{theorem}\label{Tglcenter}
The center of $H_c$ is isomorphic to the polynomial algebra in two variables, generated by the elements
whose leading term with respect to the filtration are $b=y_1x - x_1 y$ and $d=\tau(y_1x-x_1y)-(2ey_1y+h(y_1x+x_1y)-2fx_1x),$ and gr$(\mf{Z}(H_c))=\mf{Z}(gr(H_c))$.
\end{theorem}
We begin the proof by showing that $b, d$ freely generate the center of the associated graded algebra of
$H_c,$ which is just $\mf{U}(\mf{g}\ltimes V)$ (which from now on will be denoted by $H$). 

 For a reductive lie algebra $\mf{g}$ and its finite dimensional representation $V,$ the center of the enveloping
algebra $\mf{U}(\mf{g}\ltimes V)$ has been studied (see [R1], [R2], [P], [S]). But we wrote this proof before finding those references and
we include it for the sake of completeness. After this it will just remain to show that we may lift $d$ to a central element
of $H_c,$ since by the definition of $H_c$, $b-c$ is a central element.

Towards computing the center of $H,$ our first step is to show that
 $$\mf{Z}(H) \bigcap \Sym(V) = k[b].$$ Indeed, if we have a central element $\alpha =
\sum a_{i j i_1 j_1} x^i y^j x_1^{i_1} y_1^{j_1}$ (with all $a_{i j i_1
j_1}\in k$), then from $[\tau,\alpha]=0$ (since $H$ is graded, we assume
without loss of generality that $\alpha$ is homogeneous in $x, y, x_1,
y_1$) we get that $i + j = i_1 + j_1$.
Also, since $[h,\alpha]=0$ one gets $i + i_1 = j + j_1$, so that $\alpha$
may be written as: $\alpha = \sum_{i=0}^n a_i (x y_1)^i (x_1 y)^{n-i}$
(for some $n$, with $a_i$ constants). We have
\[ [e, \alpha] = \sum_{i=0}^n a_i \left( i (xx_1) (xy_1)^{i-1}
(x_1y)^{n-i} + (n-i) (x_1x) (xy_1)^i (x_1y)^{n-i-1} \right) = 0, \]

\noindent and cancelling $x x_1$ yields
\[ \sum_{i=0}^n \left( a_i (xy_1)^{i-1} (x_1y)^{n-i} i + a_i (n-i)
(xy_1)^i (x_1y)^{n-i-1} \right) = 0. \]

\noindent But this means that $(i+1)a_{i+1}=-(n-i)a_{n-i}\ \forall i$,
whence there is (at most) a unique central $\alpha$ for each $n$ up to
multiplication. Thus we must have $\alpha=c b^n$, for
some constant $c$. This proves our first claim.\medskip

In what follows we will use two subalgebras  $A_1, A_2$ of $H$, where
$A_1=\mf{U}(\mf{sl_2}\ltimes (kx\oplus ky)$), and $A_2=\mf{U}(\mf{sl_2}\ltimes (kx_1\oplus ky_1)).$
Clearly both of this algebras are isomorphic to $\mf{U}(\mf{sl_2}\ltimes k^2)$ (which will be denoted
by $A$). Here we will use the following anti-involution of $A$ (defined in [K]) and the proposition proved in [KT] 
$$j(x)=y, j(h)=h, j(e)=-f.$$
\begin{prop} Stabilizer of $k^2$ in $A$ as an algebra is generated by $t, x, y$ where $t=ey^2+hxy-fx^2$ 
is a generating central element of $A.$ Stabilizer of $\mf{g}$ in $A$ as an algebra is generated by $\Delta, t.$
\end{prop}
Genereting central elements of $A_1, A_2$ will be denoted by $t_1, t_2$
thus $t_1=ey^2+hxy-fx^2, t_2=ey_1^2+hx_1y_1-fx_1^2.$

Next we show that $\mf{Z}(H)\bigcap \mf{sl_2} \cdot \mf{U}\mf{sl_2}
\cdot \Sym(V) = 0$. 
\\Let $g = \sum a_{i j i_1 j_1} x^i
y^j x_1^{i_1} y_1^{j_1}$ be an element from this intersection, with all
$a_{i j i_1 j_1} \in \mf{sl}_2 \cdot \mf{Usl}_2$.
As before, assuming homogeneity of $g$ and using $[\tau,g]=0$, $g$ may be
rewritten as $\sum_{i,j} a_{ij} x^i y^{n-i} x_1^j y_1^{n-j}.$ Since
$[g,x] = [g,y] = 0,$ we get $[\sum_i a_{ij} x^i y^{n-i}, x] = [\sum_i
a_{ij} x^i y^{n-i}, y] = 0,$ for each $j.$
Now the above Proposition  implies
that $\sum_ia_{ij}x^iy^{n-i}\in k[t_1,x,y],$ for each $j.$
Similarly $\sum _j a_{ij}x_1^jy_1^{n-j}\in
k[t_2,x_1,y_1]$ for each $i.$ Now we have 
\[ g = \sum_j \left( \sum_i \alpha_{ij} x^i y^{n-i} \right) x_1^j
y_1^{n-j} = \sum_i \left( \sum_j \alpha_{ij} x_1^j y_1^{n-j} \right) x^i
y^{n-i}. \]
   
If $t_1, t_2$ do not appear in this expression, then $g
\in k[x, y, x_1, y_1]$ and we are done. Otherwise, comparing terms 
with highest powers of $t_1, t_2$ appearing
in both sides of equality we get $t_{1}^n g_1 = t_{2}^n g_2,$ for
some $g_1, g_2 \in k[x, y, x_1, y_1].$ Thus,
$(ey^2 + hxy - fx^2)^n g_1 = (ey_1^2 + hx_1y_1 - fx_1^2)^n g_2.$
Comparing coefficients in front of $e^n, f^n,$ one gets
$y^{2n} g_1 = y_1^{2n} g_2,\ x^{2n} g_1 = x_1^{2n} g_2,$ which can not
happen.

Now let $g$ be an arbitrary central element, so $g=\sum_{n=0}^m\tau^ng_n,$
where $g_m\neq 0, g_n\in\mf{U}(\mf{sl}_2\ltimes V).$ We will argue by induction on $n$ (degree of $g$ in $\tau$)
that $g\in k[b, d].$ So far we have shown this for $n=0.$ We may assume that elements $g_i$ are homogeneous in $x, x_1, y, y_1.$ It is clear that $g_n$ being the top coefficient of $g$ must be central, therefore we may assume
that $g_n=b^m$ for some $m.$ If $n\leq m,$ then $g-d^nb^{m-n}$ is central and has degree $<n$ in $\tau,$ 
thus $g-d^nb^{m-n}\in k[b, d],$ which gives $g\in k[d, b].$ If $n>m,$ then $d^n-b^{n-m}g$ is a central element of
degree $<n$ in $\tau,$ so $b^{n-m}g\in k[b, d].$ The latter clearly implies that $g\in k[d, b].$

In what follows we make an extensive use of computations of commutators
of the form $[\alpha, x], [\beta, y]$ where $\alpha\in\mf{Z}(\mf{U}\mf{g})$ ([T1]).
Given that $[\alpha, x]$ is an element of $\mf{U}\mf{g}x\oplus\mf{U}\mf{g}y$ and commutes with
$e$ and has a weight $1$ (with respect to ad($h$) action), it is clear that we must have
$$[\alpha, x]=(2hF(\alpha)+G(\alpha))x+4eF(\alpha)y$$ 
$$[\alpha, y]=(-2hF(\alpha)+G(\alpha))y+4fF(\alpha)x,$$
Where $F, G$ are certain linear endomorphism of $\mf{Z}(\mf{U}\mf{g}).$ We would like to establish there
properties. First, let us recall the following computation from [T1]

\begin{lemma}
We have the following identities for arbitrary $\beta \in\mf{Z}(\mf{U}\mf{g})$
$$F(\Delta\beta)=\beta+(\Delta-1)F(\beta)-G(\beta)$$
$$G(\Delta\beta)=-3\beta-4F(\beta)\Delta+(\Delta+3)G(\beta)$$
\end{lemma}
\begin{proof}We have
\begin{eqnarray*}
[\Delta\beta, x]&=&\beta ((2h-3)x+4ey)+(2hF(\beta)+G(\beta))x\Delta+4fF(\beta)y\Delta\\
&=& (2h(\beta+(\Delta-1)F(\beta)-G(\beta))+-3\beta-4F(\beta)\Delta\\
&+&(\Delta+3)G(\beta))x+4f(\beta+(\Delta-1)F(\beta)-G(\beta))y. 
\end{eqnarray*}

Thus we are done.
\end{proof}
For any $\psi(\tau)\in k[\tau],$ we have the following commutator formulas
$$[\psi(\tau), x]=\psi'(\tau)x, [\psi(\tau), y]=\psi'(\tau)y,$$ 
where $\psi'(\tau)$ denotes $\psi(\tau)-\psi (\tau-1).$

Linear endomorphisms $F, G$ are related to each other in the following way.
\begin{lemma} For any $\alpha\in \mf{Z}(\mf{U}\mf{g})$ there exist an element $\beta\in \mf{Z}(\mf{U}\mf{g})$
such that $F(\beta)=\alpha,$ also $F(\alpha)=0$ if and only if $\alpha\in k[\tau]$ and the following equality holds
 $$G(F(\alpha))=F(G(\alpha))+2F(F(\alpha)).$$

\end{lemma}
\begin{proof}
At first, we will prove the lemma when $\alpha\in k[\Delta]$. In this case we have that $[\alpha, x], [\alpha, y]\in A$ and
$[\alpha, x]=(2hF(\alpha)+G(\alpha))x+4eF(\alpha)y,$ after applying the anti-involution $j$ to
$[\alpha, x]$ we get
\begin{eqnarray*}
[\alpha, y]&=& -y(2hF(\alpha)+G(\alpha))+x4fF(\alpha)\\
&=&4fF(\alpha)x+(-2hF(\alpha)-G(\alpha))y-[4fF(\alpha), x]+\\
&&[2fF(\alpha)+G(\alpha), y]\\
&=&4fF(\alpha)x+(-2hF(\alpha)+G(\alpha))y.
\end{eqnarray*}
Hence, $2G(\alpha)y=[2hF(\alpha)+G(\alpha), y]-[4fF(\alpha), x].$ We have

\[[4fF(\alpha), x]=4F(\alpha)y+4f[F(\alpha),x]-4[F(\alpha), y] \]
\[[2hF(\alpha)+G(\alpha), y]=[G(\alpha), y]-2F(\alpha)y+(2h+2)[F(\alpha), y].\]

So,
\[2G(\alpha)y=[G(\alpha), y]-6F(\alpha)y+2h[F(\alpha), y]+6[F(\alpha), y]-4f[F(\alpha, x].\]

Equating similar terms we get,
\[4fF(G(\alpha))+2h4fF(F(\alpha))+6\times 4fF(F(\alpha))-4f(2hF(F(\alpha))+G(F(\alpha))=0,\]
so $G(F(\alpha))=F(G(\alpha))+2F(F(\alpha)).$
It is easy to check that $F(\Delta^n)$ is a polynomial of degree $n$ in $\Delta$, thus $F$ is surjective
when restricted to $k[\Delta].$
Now let us consider $\alpha=\psi (\tau)\beta$, where $\psi$ is a polynomial in $\tau$ and $\beta\in k[\Delta].$
We have 
\begin{eqnarray*}
[\psi(\tau)\beta, x] &=& \psi(\tau)[\beta, x]+\psi'(\tau)x\beta\\
&=& (\psi(\tau)-\psi'(\tau))[\beta, x]+\psi'(\tau)\beta x.
\end{eqnarray*}
Thus,
\begin{eqnarray*}
F(\psi(\tau)\beta)&=&\psi(\tau-1)F(\beta),\\ 
G(\psi(\tau)\beta)&=&\psi(\tau-1)G(\beta)+\psi'(\tau)\beta.
\end{eqnarray*}
Then we have
\begin{eqnarray*}
G(F(\psi(\tau)\beta))&=& G(\psi(\tau-1)F(\beta)=\\
&& \psi(\tau-2)G(F(\beta))+\psi(\tau-1)'F(\beta),
\end{eqnarray*}
\begin{eqnarray*}
F(G(\psi(\tau)\beta))&=& F(\psi(\tau-1)G(\beta)+\psi'(\tau)\beta)=\\
&& \psi(\tau-2)F(G(\beta))+\psi'(\tau-1)F(\beta)
\end{eqnarray*}
and since $\psi(\tau-1)'=\psi(\tau-1)$ and $F(F(\psi(\tau)\beta))=\psi(\tau-2)F(F(\beta))$ we get that
$$G(F(\psi(\tau)\beta))=F(G(\psi(\tau)\beta))+2F(F(\psi\beta)).$$
Since elements of type $\psi(\tau)\beta$ span $\mf{Z}(\mf{U}\mf{g})$, we are done with the formula. We also get
that $F$ is an epimorphism of $\mf{Z}(\mf{U}\mf{g}),$ and as $F(\alpha)$ has degree 1 less that $\alpha$ in $\Delta,$
 thus $F(\alpha)=0$ implies that $\alpha\in k[\tau].$

\end{proof}
Now, our goal is to show that there exist $\alpha\in\mf{Z}(\mf{Ug})$ such that $[d, x]=[\alpha, x].$ Then $d-\alpha$
will be a central element since  after applying the anti-involution $\eta$ to
$[d-\alpha, x]=0$ gives $[d-\alpha, y_1]=0.$ This together with the fact that
$d$ commutes with $\mf{g}$ implies that $d-\alpha\in \mf{Z}(H_c).$ At first, we will get a condition on $c$ 
which is a necessary condition for $H_c$ to satisfy a PBW property.We have 
$$[y_1x-x_1y, x]=[y_1, x]x-[x_1, x]y=[c, x].$$
So
$$[y_1, x]=2hF(c)+G(c), [x_1, x]=-4eF(c), [y_1, y]=4fF(c).$$
Thus, 
\begin{eqnarray*}
&&[[y_1, x], y]=[2hF(c)+G(c), y]=\\
&&(-2hF(G(c))+G(G(c)))y+4fF(G(c))x-2F(c)y+\\
&&(2h+2)(-2F(F(c))h+G(F(c)))y+4fF(F(c))x),
\end{eqnarray*}
on the other hand
\begin{eqnarray*}
&&[y_1, x], y]= [[y_1, y], x]=[4fF(c), x]]=4(F(c)y+[F(c), x]f)=\\
&&4(F(c)y+f[F(c), x]-[F(c)y])=4(f(2hF(F(c))x+G(F(c))x)+\\
&&4eF(F(c))y)-(-2hF(F(c))+G(F(c)))y+4fF(F(c))x)+F(c)y).
\end{eqnarray*}
Equating coefficients of $y$ we get
\begin{eqnarray*}
-2hF(G(c))+G(G(c))-2F(c)+(2h+2)(-2h)F(F(c))+\\
(2h+2)G(F(c))=4(4feF(F(c))+2hF(F(c))-G(F(c))+F(c)).
\end{eqnarray*}
Grouping them together yields
\begin{eqnarray*}
(16fe+4h^2+8h+4h)F(F(c))+2F(c)-2G(F(c))-4G(F(c))\\
-G(G(c))+2hF(G(c))-2hG(F(c))+4F(c)=0,
\end{eqnarray*}
which implies that
$$4\Delta F(F(c))+6F(c)-6G(F(c))-G(G(c))=0$$
We will refer to this equality as the Jacobi condition.

Now we want to compute $[d, x].$ Recall that $$d=\tau(y_1x-x_1y)-(2ey_1y+h(x_1y+y_1x)-2fx_1x).$$ We have
\begin{eqnarray*}
[d, x]= x(y_1x-x_1y)+\tau [c, x]-\\
(2e[y_1, x]y+x(x_1y+y_1x)+h([x_1, x]y+[y_1, x]x)-2yx_1x-2f[x_1, x]x).
\end{eqnarray*}
Thus,
\begin{eqnarray*}
[d, x]&=& xy_1x-xx_1y -xx_1y -xy_1x+2yx_1x+\tau[c, x]\\
&& -2e[y_1, x]y-h([x_1, x]y+[y_1, x]x)+2f[x_1, x]x\\ 
&=& 2(yx_1x-xx_1y)+\tau[c, x]-2e[y_1, x]y\\
&& -h([x_1, x]y+[y_1, x]x)+2f[x_1, x].
\end{eqnarray*}
Recall that 
$$[\frac{1}{2}\Delta, x]=(h-\frac{3}{2})x+2ey, [\frac{1}{2}\Delta, y]=(-h-\frac{3}{2})y+2fx.$$
So,
\begin{eqnarray*}
[\frac{1}{2}\Delta, [c, x]]=[\frac{1}{2}\Delta, [y_1, x]x-[x_1, x]y]\\
= [y_1, x]((h-\frac{3}{2})+2ey)-[x_1, x]((-h-\frac{3}{2})y+2fx)\\
=[y_1, x]hx+[y_1, x]2ey+[x_1, x]hy+\frac{3}{2}[x_1, x]y-2[x_1, x]fx-\frac{3}{2}[y_1, x]x\\
=-\frac{3}{2}[c, x]+h[y_1]x+2e[y_1, x]y+\\
h[x_1, x]y-2f[x_1, x]x-2[x_1, x]y-2[x_1, x]y+2[y_1, x]x+2[x_1, y]x.
\end{eqnarray*}
So,
\begin{eqnarray*}
[t, x]+[\frac{1}{2}\Delta, [c, x]]&=&\tau[c, x]-\frac{3}{2}[c, x]+2(yx_1x-xx_1y)\\
&&-4[x_1, x]y+2[y_1, x]x+2[x_1, y]x,
\end{eqnarray*}
we have
\begin{eqnarray*}
&& yx_1x-xx_1y-2[x_1, x]y+[y_1, x]x+[x_1, y]x\\
&=&yx_1x-xx_1y-2x_1xy+2xx_1y+x_1yx-yx_1x+[y_1, x]x\\
&=&[y_1, x]x-[x_1, x]y=[c, x].
\end{eqnarray*}
Thus,
\[[d, x]+[\frac{1}{2}\Delta, [c, x]]=\tau[c, x]+[\frac{1}{2}c, x]=[\tau c, x]-cx+[\frac{3}{2}c, x].\]

Hence $$[d, x]=[\tau +\frac{3}{2}c, x]-(cx+\frac{1}{2}[\Delta, [c, x]]).$$
We have
$$2c+[\Delta, [c, x]]=\Delta[c,x]+c[\Delta, x]+2cx-[\Delta c, x].$$
Thus, it would suffice to show that there exists $\alpha\in \mf{Z}(\mf{U}\mf{g})$ such that
$$2cx+\Delta[c, x]+c[\Delta, x]=[\alpha, x]$$
We have
\begin{eqnarray*}
2cx+c[\Delta, x]+\Delta[c, x]=\\
2cx+c(2h-3)x+4ecy+(2h\Delta+\Delta G(c))x+4e\Delta F(c)y=\\
(2h(c+\Delta F(c))+\Delta G(c)-c)x+4e(c+\Delta F(c))y.
\end{eqnarray*}
Now let us choose $\alpha\in \mf{Z}(\mf{U}\mf{g})$ such that $F(\alpha)=c+\Delta F(c).$
If we could show that $G(\alpha)-\Delta G(c)+c\in k[\tau]$ then we would be done. 
But recall that $G(F(\alpha))=F(G(\alpha))+2F(F(\alpha))$, thus our goal is equivalent to showing that
$$G(F(\alpha))-2F(F(\alpha))=F(\Delta G(c))-F(c).$$ Now, recall that
$$F(\Delta\beta)=\beta+(\Delta-1)F(\beta)-G(\beta)$$
$$G(\Delta\beta)=-3\beta-4F(\beta)\Delta+(\Delta+3)G(\beta).$$
Thus,
$$F(\Delta G(c))=G(c)+(\Delta-1)F(G(c))-G(G(c))$$
$$F(\Delta F(c))=F(C)+(\Delta-1)F(F(c))-G(F(c))$$
$$G(\Delta F(c))=-3F(c)-4F(F(c))\Delta+(\Delta+3)G(F(c)).$$
Thus the desired equality that we need to prove becomes
\begin{eqnarray*}
G(c)-7F(c)-4\Delta F(F(c))+(\Delta+3)G(F(c))\\
-2(\Delta-1)F(F(c))+2G(F(c))\\
=G(c)+(\Delta-1)F(G(c))-G(G(c))-F(c).
\end{eqnarray*}
Putting all terms on one side we get
\begin{eqnarray*}
6F(c)+4\Delta F(F(c))-2G(F(c))-G(G(c))+\\
(\Delta-1)(2F(F(c))+F(G(c)))-(\Delta+3)G(F(c))=0,
\end{eqnarray*}

which is nothing but the Jacobi condition. To summarize, we have shown that there exists $\alpha\in \mf{Z}(\mf{Ug}$)
such that $F(\alpha)=c+\Delta F(c)$ and $d-(\tau+\frac{3}{2}c)-\frac{1}{2}\alpha$ is a central element.

Now we will discuss a little bit of representation theory of $H_c$. As usual, one does this withing
the apropriately defined analog of the BGG category $\mathcal O$ for semi-simple Lie algebras. 
\begin{defin}The category $\mathcal O$ for an algebra $H_c$ is define as a full subcategory of the category of
finitely generated left $H_c$-modules
on which $h, \tau$ act diagonalizably and $e, x, y$ act locally nilpotently.
\end{defin}
Thus real parts of eigenvalues of $\tau$ are bounded from above. As usual, if $M$ belongs to the category
$\mathcal O$ then element $v\in M$ is called a maximal vector if and only if $ev=xv=yv=0.$ We also have a standard definition of
the Verma module for any pair of weights $\lambda, \mu$ (of $h, \tau$ respectively): 
$M(\lambda, \mu)=H_c\otimes_{B}k_{\lambda, \mu},$  where $B$ is a
subalgebra of $H_c$ generated by $e, x, y, h, \tau,$ and $k_{\lambda, \mu}=kv$ is its  one dimensional representation
on which $h, \tau$ act like multiplication by $\lambda, \mu.$ Clearly 
$M(\lambda, \mu)=\mf{U}(kf\ltimes (k^2)^{*})v.$ Standard argument shows that $M(\lambda, \mu)$ has a 
unique simple quotient, which will be denoted by $L(\lambda, \mu).$ 

We have a spectral decomposition of the category $O$: for any character $\chi\in Spec(\mf{Z}(H_c))$ we define
$O^{\chi}$ to be a full subcatogory of modules on which ker$(\chi)$ acts nilpotently, then  the category
$O$ decomposes into a direct sum of blocks $O^{\chi}.$

Next we would like to determine conditions of $\lambda, \mu$ which gives finite dimensional $L(\lambda, \mu)$.
So let us assume that $L(\lambda, \mu)$ is finite dimensional. Let us denote by $L^n$ weight subspace for $\tau$
corresponding to the weight $n\in k.$ Clearly there are only finitely many $n$ such that $L^n\neq 0$, and $\mu$ is
the biggest of them all (it is also clear that any two $n$ differ by an integer and their are no gaps between them) and $L^n$ 
is a $\mf{g}$-module. Then clearly there exists $m>0$ such that $f^mvL^{\mu}=x_1^mL^{\mu}=y_1^mL^{\mu}=0.$ It
is clear that $L^{\mu}$ is an irreducible $\mf{g}$ module. 
Let us write $[y^m, x_1^m]=\alpha_m$ mod$(H_cV)$ for some well defined $\alpha_m\in \mf{U}\mf{g}$ (which only depends on, $c$ and of course $n$). Thus we have that $\alpha_mL^{\mu}=0.$ Now we claim that all these conditions are sufficient for finite dimensionality.

 Indeed, at first we claim that $x^iy^{m-i}x_1L^{\mu}=0$ for all $i.$ Indeed, proceeding by induction on $i$
$$ey^ix^{m-i}x_1v-y^ix^{m-i}x_1ev=iy^{i-1}x^{m-i+1}x_1^mv=0$$
now this implies that $H_cx_1^{m}L^{\mu}\cap L^{\mu}=0$ thus $x_1^mL^{\mu}=0.$ This implies that
$y_1^ix_1^{m-i}L^{\mu}=0,$ indeed arguing by induction on $i,$ 
\\we have $(fx_1^iy_1^{m-i}-x_1^iy_1^{m-i}f)L^{\mu}=0$ thus $ix_1^{i-1}y_1^{m-i+1}L^{\mu}=0$ hence we are done. So we have the following
\begin{prop} Irreducible module $L(\lambda, \mu)$ is finite dimensional if and only if $\lambda$ is a nonegative integer and there exists $m$ such that  $\alpha_m V(\lambda, \mu)=0$ where $V(\lambda, \mu)$ is an irreducible
representation of $\mf{g}$ of heighest weight $\lambda, \mu$.

\end{prop}
Let us briefly discuss the case when the ground field has positive characteristic. We have the following general result
\begin{prop} Let $H_{\gamma}$ be any infinitesimal Cherednik algebra over the ground field $k$ of positive characteristic.
Then $H_{\gamma}$ is a prime Noetherian ring which is Auslander regular and Cohen-Macaulay 
and it is a finitely generated module over its center.
\end{prop}
\begin{proof} The associated graded of $H_{\gamma}$ is a Hopf algebra of finite global
dimension which is finite dimensional over the center, after applying general results
of [BG], we get that $H_{\gamma}$ is a prime Noetherian ring, which is also Auslander regular and Cohen-Macaulay.
Thus it remains to show that $H_{\gamma}$ is a finitely generated module over the center.

By the definition of $H_{\gamma},$ algebras $\mf{U}(\mf{g}\ltimes k^n),\mf{U}(\mf{g}\ltimes (k^n)^{*})$ are its subalgebras, therefore there exists big enough $n,$ such that all restricted powers $g^{[p^n]}$ are central for any $g\in\mf{g}.$ Now we claim that for any $v\in k^n,$ $v^{p^2}$ belongs to the center of $H_{\gamma}.$ Indeed, it is clear that this elemen commutes
$\mf{g},$ thus we just need to show that $[v^{p^2}, w^{*}]=0$ for all $w^*\in (k^n)^{*}.$ We have 
$[v^{p^2}, w^{*}]=ad(v^p)^p(w^*)=ad(v^p)^{p-1}([v^{p},w^*]),$ but since $[v^{p}, w^{*}]=ad(v)^p(w^{*})\in \mf{U}(\mf{g}\ltimes k^n),$ we get that $[v^{p^2}, w{^*}]=0.$ Therefore $(k^n)^{p^2}\in \mf{Z}(H_c),$ similarly
$((k^n)^{*})^{p^2}\in \mf{Z}(H_c),$ therefore we may conclude that $H_{\gamma}$ is finite dimensional over its center.

\end{proof}

\end{document}